\documentclass[11pt,letterpaper]{article}

\usepackage{amsmath}
\usepackage{amssymb}
\usepackage{amsthm}
\usepackage[hyperfootnotes=false]{hyperref}
\usepackage[inline]{enumitem}
\usepackage[margin=1.4in]{geometry}
\newcommand{\abs}[1]{\lvert#1\rvert} 

\let\phi\varphi

\let\setminus\smallsetminus

\theoremstyle{plain}
\newtheorem{thm}{Theorem}[section]
\theoremstyle{definition}
\newtheorem{lem}[thm]{Lemma}
\newtheorem{prop}[thm]{Proposition}
\theoremstyle{definition}
\newtheorem*{rmk}{Remark}
\newtheorem{dfn}[thm]{Definition}

\title{Supercharacter Theories of Dihedral Groups}
\author{Jonathan Lamar}
\date{}

\allowdisplaybreaks

\begin{document}


\maketitle
\abstract{The set of supercharacter theories of a fixed group $G$ forms a natural lattice.  An open question in the study of supercharacter theories is to classify this lattice, and to date, this has only been done for the cyclic groups $\mathbb{Z}_n$.  In this paper, we classify the supercharacter theory lattice of the dihedral groups $D_{2n}$ in terms of their cyclic subgroups of rotations.}


\section{Introduction}

Supercharacter theories are certain coarsenings of the usual character theory of a finite group.  They were first introduced in \cite{Andre1995}, \cite{Yan2001}, \cite{Andre2002}, and \cite{Andre2006} to study the representation theory of the groups $U_n(q)$ of $n\times n$ unipotent upper triangular matrices over finite fields $\mathbb{F}_q$.  In these papers, the authors studied a particular supercharacter theory of $U_n(q)$ which approximates the character theory of those groups.  Although the character theory of $U_n(q)$ remains a ``wild'' problem (see \cite{Higman1960}, \cite{VeraLopez2003}, \cite{Poljak1966}), this approximation has yielded partial solutions to previously intractible problems.  For example, in \cite{AriasCastro2004}, the authors used this supercharacter theory to provide bounds on the rate of convergence to equilibrium of a well-known random walk on $U_n(q)$.

In \cite{Diaconis2008}, Diaconis and Isaacs formally defined the notion of a supercharacter theory for an arbitrary finite group and studied a particular supercharacter theory for a family of finite groups known as algebra groups.  These and related supercharacter theories were studied in subsequent papers (for example, see \cite{Diaconis2009}, \cite{Marberg2009}, \cite{Aguiar2012}, and \cite{Andrews2015c}) where priority was placed on using supercharacter theory to ease computation of character values and build combinatorial Hopf-theoretic structures for nested families of groups with given supercharacter theories in a manner analogous to \cite{macdonald}.

As we will see below, the set $\mathrm{SCT}(G)$ of all supercharacter theories of a fixed group $G$ forms a lattice under a natural partial order.  The first author to directly attempt to classify the lattice of all supercharacter theories of a fixed group was Hendrickson in \cite{Hendrickson2008}.  In this paper, he classified the supercharacter theories of cyclic groups of prime order, while in \cite{Hendrickson2012}, he classified the supercharacter theories of arbitrary cyclic groups using the previous work of Leung and Man (see \cite{Leung1996} and \cite{Leung1998}) on Schur rings.  In \cite{Benidt2014}, the authors studied the combinatorial properties of the lattice of supercharacter theories of a cyclic group.

So far, no attempt has been made to classify these lattices for other families of finite groups.  Since cyclic groups' supercharacter theories have been classified, it is natural to turn our attention to extensions of cyclic groups.  The goal of the present paper is to provide an explicit classification of the supercharacter theories of the dihedral groups $D_{2n}$ of order $2n$ using their cyclic subgroups of rotations.  We will do this by first embedding a canonical sublattice of $\mathrm{SCT}(\mathbb{Z}_d)$ into $\mathrm{SCT}(D_{2n})$ for each divisor $d$ of $n$, and defining a certain refining map which produces the remainder of $\mathrm{SCT}(D_{2n})$ from the supercharacter theories obtained from $\mathrm{SCT}(\mathbb{Z}_n)$.

\section{Basic notions}

In this section, we will summarize some of the known results concerning the lattice of all supercharacter theories of a group.  In particular, Lemma \ref{lem:Hen_sums_span}, Propositions \ref{prop:Hen_order} and \ref{prop:Hen_join}, and Theorem \ref{thm:Zn_classification} are reproduced from \cite{Hendrickson2008} and \cite{Hendrickson2012}.

\subsection{The lattice of supercharacter theories}\label{subsec:Lattice}

A \emph{supercharacter theory} of a finite group $G$ is an ordered pair $S = (\mathcal{K},\mathcal{X})$, where $\mathcal{K}$ is a partition of $G$ into unions of conjugacy classes, $\mathcal{X}$ is a partition of $\mathrm{Irr}(G)$, and such that the following conditions are met:
\begin{enumerate}
	\item $\abs{\mathcal{K}} = \abs{\mathcal{X}}$;
	\item if $e$ denotes the identity of $G$, then $\{e\} \in \mathcal{K}$;
	\item for each $X\in \mathcal{X}$, the character $\sigma_X = \sum_{\chi\in X}\chi(e)\chi$ is constant on the parts of $\mathcal{K}$.
\end{enumerate}
If $S=(\mathcal{K},\mathcal{X})$ is a supercharacter theory, then we call the parts of $\mathcal{K}$ the \emph{superclasses} of $S$, or \emph{$S$-superclasses}, and we call the the functions $\sigma_X$ the \emph{supercharacters of $S$}, or \emph{$S$-supercharacters}.  Let $\abs{S}$ denote the number of superclasses (equivalently the number of supercharacters) of $S$.

For any group $G$, let $\mathbf{1}_G$ denote the trivial character of $G$ and let $e_G$ denote the identity element of $G$.  If there is no chance of confusion, we will drop the subscript and simply write $e$ for the identity.  For any group $G$, we can define two extreme supercharacter theories, which are (using Hendrickson's notation)
\[
	m(G) = \big(\mathrm{Cl}(G), \big\{\{\chi\}\,:\,\chi\in\mathrm{Irr}(G)\big\}\big),
\]
and
\[
	M(G) = \big(\big\{\{e\},G\setminus\{e\}\big\}, \big\{\{\mathbf{1}_G\}, \mathrm{Irr}(G)\setminus\{\mathbf{1}_G\}\big\}\big),
\]
where $\mathrm{Cl}(G)$ is the set of conjugacy classes of $G$.  These supercharacter theories are distinct for all groups $G$ of order greater than 2.  In \cite{Burkett2017}, it is shown that the only groups for which these are all of the supercharacter theories are the cyclic group $\mathbb{Z}_3$, the symmetric group $S_3$, and the finite symplectic group $\mathrm{Sp}(6,2)$.

For an algebraic interpretation of $\mathrm{SCT}(G)$, consider the following.  To every supercharacter theory $S=(\mathcal{K},\mathcal{X})$, there is an associated algebra $\mathrm{scf}_S(G)$ of \emph{superclass functions}, i.e., complex-valued functions on $G$ which are constant on the parts of $\mathcal{K}$.  It is clear that the superclass identifier functions form a basis for $\mathrm{scf}_S(G)$.  Since the supercharacters are linearly independent and $\abs{X}=\abs{K}$, it follows that the supercharacters also form a basis for $\mathrm{scf}_S(G)$.

Through the identification $f\leftrightarrow \sum_{g\in G}f(g)g$ of class functions on $G$ with elements of $Z(\mathbb{C} G)$, we obtain a subalgebra of $Z(\mathbb{C} G)$.  An obvious basis for this subalgebra is the set of \emph{superclass sums}, which are defined to be the elements
\[
	\underline{K} = \sum_{g\in K}g \in \mathbb{C} G
\]
for $K\in\mathcal{K}$.  In \cite{Diaconis2008}, Diaconis and Isaacs proved that the superclass sums and supercharacters span their respective generated subalgebras and that the partitions $\mathcal{K}$ and $\mathcal{X}$ uniquely determine each other.

\begin{lem}\label{lem:Hen_sums_span}\cite[Lemma 2.1]{Hendrickson2012}, \cite[Theorem 2.2(b)]{Diaconis2008}
	Let $\mathcal{K}$ be a partition of $G$.  Then $\mathcal{K}$ is the superclass partition of a supercharacter theory if and only if $\mathrm{span}\{\underline{K}:K\in\mathcal{K}\}$ is a subalgebra of $Z(\mathbb{C} G)$.
\end{lem}

Let $\mathrm{SCT}(G)$ denote the set of all supercharacter theories of the fixed group $G$ and recall the refinement partial order on set partitions: if $\mathcal{A}$ and $\mathcal{B}$ are partitions of the same set $S$, then we say $\mathcal{A}\leq\mathcal{B}$ if each part of $\mathcal{A}$ is contained in a part of $\mathcal{B}$, or equivalently, if each part of $\mathcal{B}$ is a union of parts of $\mathcal{A}$.  The following result is due to \cite{Hendrickson2012}, but we can provide an alternate proof.

\begin{prop}\label{prop:Hen_order}\cite[Corollary 3.4]{Hendrickson2012}
	Let $S=(\mathcal{K},\mathcal{X})$ and $T = (\mathcal{L},\mathcal{Y})$ be supercharacter theories of a finite group $G$.  Then with respect to the refinement partial orders on $G$ and $\mathrm{Irr}(G)$, $\mathcal{K}\leq\mathcal{L}$ if and only if $\mathcal{X}\leq\mathcal{Y}$.
\end{prop}

\begin{proof}
	Suppose $\mathcal{K}\leq\mathcal{L}$.  Then for all $Y\in\mathcal{Y}$, the $T$-supercharacter $\sigma_Y$ is constant on the parts of $\mathcal{K}$.  Thus, we have
	\[
		\sigma_Y = \sum_{X\in\mathcal{X}}c_X\sigma_X
	\]
	for some constants $c_X$.  In fact, by examining coefficients of irreducible characters, one sees that $c_X\in\{0,1\}$ for all $X$, whence $Y$ is a union of parts of $\mathcal{X}$.  Thus, $\mathcal{X}\leq\mathcal{Y}$.

	Conversely, suppose $\mathcal{X}\leq\mathcal{Y}$.  For any subset $A$ of $G$, let $\delta_A$ denote the indicator function of that set.  Then because the $T$-supercharacters form a basis for $\mathrm{scf}_T(G)$, it follows that for any $L\in\mathcal{L}$, there exist constants $c_Y$ such that
	\begin{align*}
		\delta_L &= \sum_{Y\in\mathcal{Y}}c_Y\sigma_Y \\
		&= \sum_{Y\in\mathcal{Y}}\sum_{X\subseteq Y}c_Y\sigma_X \\
		&= \sum_{Y\in\mathcal{Y}}\sum_{X\subseteq Y}\sum_{K\in\mathcal{K}}c_Y\sigma_X(K)\delta_K.
	\end{align*}
	Thus, $\delta_L$ is constant on the parts of $\mathcal{K}$, which implies that $L$ is a union of parts of $\mathcal{K}$.  Therefore, $\mathcal{K}\leq\mathcal{L}$.
\end{proof}

Using the above proposition, we may define a partial order on $\mathrm{SCT}(G)$ as follows: if $S = (\mathcal{K},\mathcal{X})$ and $T = (\mathcal{L},\mathcal{Y})$ are supercharacter theories of $G$, say $S\leq T$ if $\mathcal{K}$ is a refinement of $\mathcal{L}$, or equivalently if $\mathcal{X}$ is a refinement of $\mathcal{Y}$.  For $S,T\in\mathrm{SCT}(G)$, we have $S\leq T$ if and only if $\mathrm{scf}_T(G)\subseteq\mathrm{scf}_S(G)$.  Note that $S\leq T$ certainly implies $\abs{T}\leq\abs{S}$, but the converse need not be true.  Similarly, if $S\leq T$ and $\abs{T} = \abs{S}-1$, then $S<T$ is a covering relation in $\mathrm{SCT}(G)$.

\begin{prop}\label{prop:Hen_join}\cite[Proposition 3.3]{Hendrickson2012}
	If $S = (\mathcal{K},\mathcal{X})$ and $T = (\mathcal{L},\mathcal{Y})$ are supercharacter theories of a group $G$, then their lattice-theoretic join $S\vee T$ in $\mathrm{SCT}(G)$ exists, and moreover, it is given by
	\[
		S\vee T = (\mathcal{K}\vee\mathcal{L}, \mathcal{X}\vee\mathcal{Y}).
	\]
\end{prop}

Since $\mathrm{SCT}(G)$ is finite, the above result induces a meet operation on $\mathrm{SCT}(G)$ as well, i.e., we define $S\wedge T = \bigvee_{U\leq S,T}U$, making $\mathrm{SCT}(G)$ a lattice.  Unfortunately, the meet of two supercharacter theories need not be their mutual refinement.  In fact there are examples of meets of the form
\[
	(\mathcal{K},\mathcal{X})\wedge(\mathcal{L},\mathcal{Z}) = (\mathcal{M},\mathcal{Z}),
\]
where one of $\mathcal{M},\mathcal{Z}$ is equal to the appropriate mutual refinement and the other is not.  However, the superclasses and supercharacters of $(\mathcal{K},\mathcal{X})\wedge(\mathcal{L},\mathcal{Y})$ are refinements of $\mathcal{K}\wedge\mathcal{L}$ and $\mathcal{X}\wedge\mathcal{Y}$, respectively.

\subsection{Actions on supercharacter theories}\label{subsec:Actions}

If $A$ is a subgroup of $\mathrm{Aut}(G)$ for a finite group $G$ (more generally, if $A$ is a group which acts on $G$ via automorphisms), then the natural actions of $A$ on $G$ and on $\mathrm{Irr}(G)$ define an action of $A$ on $\mathrm{SCT}(G)$ in the following way.  If $S = (\mathcal{K},\mathcal{X})$ is a supercharacter theory and $\alpha$ an automorphism, then one can check that the pair of partitions
\[
	\alpha\cdot\mathcal{K} = \big\{\{\alpha(g)\,:\,g\in K\}\,:\,K\in\mathcal{K}\big\}
\]
and
\[
	\alpha\cdot\mathcal{X} = \big\{\{\chi\circ\alpha^{-1}\,:\,\chi\in X\}\,:\,X\in\mathcal{X}\big\}
\]
form a new supercharacter theory, which we denote $\alpha\cdot S$.  If $S$ is fixed by this action, then $S$ is called \emph{$A$-characteristic}, and if $A = \mathrm{Aut}(G)$, we simply say $S$ is \emph{characteristic}.  The subset of characteristic supercharacter theories is denoted $\mathrm{CharSCT}(G)$, and the following lemma tells up that it is in fact a sublattice of $\mathrm{SCT}(G)$.

\begin{lem}\label{prop:char_is_lattice}
	Let $G$ be a finite group and let $S=(\mathcal{K},\mathcal{X})$ and $T=(\mathcal{L},\mathcal{Y})$ be characteristic supercharacter theories of $G$.  Then $S\wedge T$ and $S\vee T$ are characteristic.
\end{lem}

\begin{proof}
	Write $S\wedge T = (\mathcal{M},\mathcal{Z})$, and let $M\in\mathcal{M}$.  Then there exist $K\in\mathcal{K}$ and $L\in\mathcal{L}$ such that $M\subseteq K\cap L$.  For any $\alpha\in\mathrm{Aut}(G)$, we have
	\begin{align*}
		\alpha\cdot(K\cap L) &= \{\alpha(g)\,:\,g\in K\cap L\} \\
		&= \{\alpha(g)\,:\,g\in K\}\cap\{\alpha(g)\,:\,g\in L\} \\
		&= (\alpha\cdot K)\cap(\alpha\cdot L),
	\end{align*}
	hence $\alpha\cdot M\subseteq (\alpha\cdot K)\cap(\alpha\cdot L)$.  Thus, $\alpha\cdot(S\wedge T) \leq S,T$, and therefore $\alpha\cdot(S\wedge T)\leq S\wedge T$.  But $\lvert S\wedge T\rvert=\lvert \alpha\cdot(S\wedge T)\rvert$, which implies that $\alpha\cdot(S\wedge T)=S\wedge T$, and therefore $S\wedge T$ is characteristic.

	Now write $S\vee T=(\mathcal{N},\mathcal{W})$, and let $N\in\mathcal{N}$.  Then there exist parts $K_i\in\mathcal{K}$ and $L_j\in\mathcal{L}$ such that $N=\bigcup_iK_i=\bigcup_jL_j$.  Hence for any $\alpha\in\mathrm{Aut}(G)$,
	\[
		\alpha\cdot N = \bigcup_i(\alpha\cdot K_i) = \bigcup_j(\alpha\cdot L_j),
	\]
	and hence $\alpha\cdot N$ is a union of parts of $\mathcal{K}$ and of $\mathcal{L}$.  Consequently, $\alpha\cdot(S\vee T)\geq S,T$, and we conclude as above that $\alpha\cdot(S\vee T)=S\vee T$.
\end{proof}

There is a stronger property than that of being characteristic, which we will need in order to distinguish which supercharacter theories of a normal subgroup can be extended to the larger group.  Let $A$ be a group which acts on a finite group $G$ via automorphisms.  Then (see \cite[6.32ff]{isaacschar}), the actions of $A$ on $G$ and $\mathrm{Irr}(G)$ contain the same number of orbits.  More generally, if $S$ is any characteristic supercharacter theory of $G$, then $A$ permutes the superclasses and supercharacters of $S$ in such a way that the number of orbits of superclasses and of supercharacters are equal.  If $A$ fixes each superclass (equivalently, $A$ fixes each supercharacter), then we call $S$ an \emph{$A$-invariant} supercharacter theory, and we write $A\mathrm{InvSCT}(G)$ for the set of all $A$-invariant supercharacter theories of $G$.  There is a unique minimal element of $A\mathrm{InvSCT}(G)$, denoted $m_A(G)$, whose superclass and supercharacter partitions are the orbits of the actions of $A$ on $G$ and $\mathrm{Irr}(G)$, respectively.  Hence, the $A$-invariant supercharacter theories are simply those whose superclass and supercharacter partitions are unions of $A$-orbits.

\subsection{Products of supercharacter theories}\label{subsec:Products}

In this section, we will discuss three different constructions for products of two supercharacter theories: the direct product, the $\ast$-product, and the $\Delta$-product.  While only the $\ast$-product will be used in the classification $\mathrm{SCT}(D_{2n})$, the others are necessary to state the classification of $\mathrm{SCT}(\mathbb{Z}_n)$.

The first and simplest product construction is the \emph{direct product}.  If $H$ and $K$ are two finite groups, $S=(\mathcal{K},\mathcal{X})\in\mathrm{SCT}(H)$ and $T=(\mathcal{L},\mathcal{Y})\in\mathrm{SCT}(K)$, then the direct product of $S$ and $T$, denoted $S\times T$, is the supercharacter theory of $H\times K$ with superclass partition
\[
	\mathcal{K}\times\mathcal{L} = \{K\times L\,:\,K\in\mathcal{K},\,L\in\mathcal{L}\}
\]
and supercharacter partition
\[
	\mathcal{X}\times\mathcal{Y} = \{X\times Y\,:\,X\in\mathcal{X},\,Y\in\mathcal{Y}\}.
\]
That $S\times T$ is a valid supercharacter theory of $H\times K$ is an easy exercise.

Next, let $G$ be a finite group with normal subgroup $N$.  Then $G$ acts on $N$ via automorphisms and the partition of $N$ into $G$-orbits is the superclass partition of the minimal $G$-invariant supercharacter theory $m_G(N)$ of $N$.  The superclasses of $m_G(N)$ are simply the $G$-conjugacy classes which are subsets of $N$.  Thus, $G$-invariant supercharacter theories of $N$ are those whose superclasses are unions of $G$-conjugacy classes.  Therefore, the most na\"{i}ve putative method of combining a $G$-invariant supercharacter theory $S$ of $N$ with a supercharacter theory $T$ of $G/N$ could be to pull back the nonidentity superclasses of $T$ and combine them with the superclasses of $S$.  It is perhaps surprising that this method does in fact produce a supercharacter theory of $G$, which Hendrickson calls the \emph{$\ast$-product} of $S$ and $T$, and denotes $S\ast_NT$.

\begin{prop}\cite[Theorem 3.5]{Hendrickson2008}
	Let $S=(\mathcal{K},\mathcal{X})$ be a $G$-invariant supercharacter theory of $N$ and $T = (\mathcal{L},\mathcal{Y})$ be a supercharacter theory of $G/N$.  Then there is a supercharacter theory of $G$ given by $S\ast_NT = (\mathcal{M},\mathcal{Z})$, whose superclass partition is
	\begin{equation}\label{eq:star_superclasses}
		\mathcal{M} = \mathcal{K} \cup \Big\{\bigcup_{gN\in L}gN\,:\,L\in\mathcal{L}\setminus\{\{eN\}\}\Big\}
	\end{equation}
	and whose supercharacter partition is
	\begin{equation}\label{eq:star_supercharacters}
		\mathcal{Z} = \big\{X^G\,:\,X\in\mathcal{X}\setminus\{\mathbf{1}_N\}\big\} \cup \mathcal{Y},
	\end{equation}
	where $X^G = \{\mathrm{Irr}(G\vert\chi):\chi\in X\}$ and the elements of $\mathrm{Irr}(G/N)$ are identified as elements of $\mathrm{Irr}(G)$ through inflation.
\end{prop}

\begin{dfn}
	Let $G$ be a finite group and let $N$ be a normal subgroup.
	\begin{enumerate}[label=(\alph{*})]
		\item We say a supercharacter theory of $G$ \emph{factors over $N$} if it can be written as a $\ast$-product of a $G$-invariant supercharacter theory of $N$ with a supercharacter theory of $G/N$.
		\item The unique maximal supercharacter theory of $G$ which factors over $N$ is $M(N)\ast_NM(G/N)$ and we denote this supercharacter theory by $MM_N(G)$.
		\item The unique minimal supercharacter theory of $G$ which factors over $N$ is $m_G(N)\ast_Nm(G/N)$ and we denote this supercharacter theory by $mm_N(G)$.
	\end{enumerate}
\end{dfn}

Finally, there is a construction in \cite{Hendrickson2008} known as the $\Delta$-product, which we will summarize here.  If $S=(\mathcal{K},\mathcal{X})$ is a supercharacter theory of a finite group $G$ and $N$ is a subgroup, we call $N$ \emph{$S$-normal} if it is a union of $S$-superclasses (or equivalently, if $\mathrm{Irr}(G/N)$ is a union of parts of $\mathcal{X}$).  In this setting, we can define two new supercharacter theories as follows.  The first is the restricted supercharacter theory $S_N\in\mathrm{SCT}(N)$, whose superclasses are merely those which lie in $N$:
\begin{equation}\label{eq:restricted_sct}
	S_N = \Big(\big\{K\in\mathcal{K}\,:\,K\subseteq N\big\},\,\big\{X_N\,:\,X\in\mathcal{X},\,X\not\subseteq\mathrm{Irr}(G/N)\big\}\cup\big\{\{\mathbf{1}_G\}\big\}\Big),
\end{equation}
where $X_N$ denotes the set of irreducible constituents of $\mathrm{Res}^G_N(\sigma_X)$.  The second is the deflated supercharacter theory $S^{G/N}\in\mathrm{SCT}(G/N)$, whose supercharacter blocks are merely those which lie in $\mathrm{Irr}(G/N)$:
\begin{equation}\label{eq:deflated_sct}\begin{split}
	S^{G/N} = \Big(&\big\{\{N\}\big\}\cup\big\{\{gN\,:\,g\in K\}\,:\,K\in\mathcal{K},\,K\not\subseteq N\big\}, \\
	&\big\{X\in\mathcal{X}\,:\,X\subseteq\mathrm{Irr}(G/N)\big\}\Big).
\end{split}\end{equation}

\begin{prop}\cite[Theorem 4.1]{Hendrickson2008}
	Let $G$ be a finite group with normal subgroups $N\mathrel{\lhd} M\mathrel{\lhd} G$, let $S=(\mathcal{K},\mathcal{X})\in G\mathrm{InvSCT}(M)$ and let $T=(\mathcal{L},\mathcal{Y})\in\mathrm{SCT}(G/N)$.  Suppose
	\begin{enumerate}[label=(\alph{*})]
		\item $N$ is $S$-normal,
		\item $M/N$ is $T$-normal, and
		\item $S^{M/N} = T_{M/N}$,.
	\end{enumerate}
	Then we can form a supercharacter theory of $G$ called the $\Delta$-product of $S$ and $T$, denoted $S\mathrel{\Delta}T=(\mathcal{M},\mathcal{Z})$, whose superclass partition is
	\[
		\mathcal{M} = \mathcal{K}\cup\bigg\{\bigcup_{gN\in L}gN\,:\,L\in\mathcal{L},\,L\not\subseteq M/N\bigg\},
	\]
	and whose supercharacter partition is
	\[
		\mathcal{Z} = \mathcal{Y}\cup\big\{X^G\,:\,X\in\mathcal{X},\,X\not\subseteq\mathrm{Irr}(M/N)\big\},
	\]
	where $X^G$ denotes the set of constituents of $\mathrm{Ind}^G_M(\sigma_X)$.
\end{prop}
One can check that if $M=N$, then the $\Delta$-prodcut reduces to the $\ast$-product.

Now that we have established the necessary terminology, we can state the classification of $\mathrm{SCT}(\mathbb{Z}_n)$.

\begin{thm}\label{thm:Zn_classification}\cite[Theorem 9.1]{Hendrickson2012}
	Let $G$ be a finite cyclic group, and let $S\in\mathrm{SCT}(G)$ be such that $S\neq M(G)$.  Then at least one of the following is true:
	\begin{enumerate}
		\item $S = m_A(G)$ for some subgroup $A$ of $\mathrm{Aut}(G)$;
		\item $S$ is of the form $(\mathcal{K}\times\mathcal{L},\mathcal{X}\times\mathcal{Y})$ for some decomposition $G = H\times K$ and supercharacter theories $(\mathcal{K},\mathcal{X})\in\mathrm{SCT}(H)$ and $(\mathcal{L},\mathcal{Y})\in\mathrm{SCT}(K)$;
		\item $S$ is a nontrivial $\Delta$-product.
	\end{enumerate}
\end{thm}

As it turns out, the $\ast$-product provides an embedding of a sublattice of $\mathrm{SCT}(\mathbb{Z}_n)$ into $\mathrm{SCT}(D_{2n})$.  We will see in subsequent sections that the image of this embedding determines the structure of $\mathrm{SCT}(D_{2n})$.

\begin{lem}\label{lem:starproductlatticeinjection}\cite[Lemma 2.2]{Benidt2014}
	Let $G$ be a group and let $N$ be a normal subgroup of $G$.  Then the map
	\[
		f : G\mathrm{InvSCT}(N)\times\mathrm{SCT}(G/N) \to \mathrm{SCT}(G); \; (S,T) \mapsto S\ast_NT
	\]
	is an injection of lattices whose image is the interval $[mm_N(G),MM_N(G)]$.
\end{lem}

The $\ast$-product also preserves the property of being characteristic.

\begin{lem}\label{lem:star_preserves_char}
	Let $N$ be a characteristic subgroup of $G$, let $S=(\mathcal{K},\mathcal{X})\in\mathrm{SCT}(N)$ and let $T=(\mathcal{L},\mathcal{Y})\in\mathrm{SCT}(G/N)$. Then $S\ast_NT$ is characteristic if and only if $S$ is $A$-characteristic and $T$ is $B$-characteristic, where $A\subseteq\mathrm{Aut}(N)$ and $B\subseteq\mathrm{Aut}(G/N)$ are the images of $\mathrm{Aut}(G)$ under the canonical maps $\mathrm{Res}^G_N:\mathrm{Aut}(G)\to\mathrm{Aut}(N)$ and $\mathrm{Proj}^G_{G/N}:\mathrm{Aut}(G)\to\mathrm{Aut}(G/N)$, respectively.
\end{lem}

\begin{proof}
	Suppose $S$ is $A$-characteristic and $T$ is $B$-characteristic.  Let $\alpha\in\mathrm{Aut}(G)$ and consider the action of $\alpha$ on $S\ast_NT$.  Since $N$ is characteristic, $\alpha$ permutes the superclasses of $S$ via the action of $\mathrm{Res}^G_N(\alpha)$.  Let $L$ be a superclass of $T$ and let $\widetilde{L} = \bigcup_{gN\in L}gN$ be its preimage under the canonical map $G\to G/N$.  Similarly, let $L'$ be the image of $L$ under $\mathrm{Proj}^G_{G/N}(\alpha)$ and let $\widetilde{L'}$ be its preimage.  We are done if we can show that $\alpha(M)=M'$.  But this is verified directly:
	\[
		\alpha(\widetilde{L}) = \alpha\bigg(\bigcup_{x\in L}xN\bigg) = \bigcup_{x\in L}\alpha(xN) = \bigcup_{x\in L}\alpha(x)N = \widetilde{L'}.
	\]
	Thus, $\mathrm{Aut}(G)$ permutes the superclasses of $S\ast_NT$, so it is characteristic.

	Now suppose $S\ast_NT$ is characteristic.  Then its superclasses are permuted.  Since the superclasses which are unions of $N$-cosets (i.e., those inherited from $\mathcal{L}$) all contain at least $\abs{N}$ elements, it follows that these are permuted amongst themselves, and the superclasses which are parts of $\mathcal{K}$ are permuted amongst themselves.  This immediately implies that $S$ is $A$-characteristic.  As before, let $\alpha\in\mathrm{Aut}(G)$ and consider the action of $\mathrm{Proj}^G_{G/N}(\alpha)$ on $\mathcal{L}$.  Let $L\in \mathcal{L}$ and let $\widetilde{L}$ be its preimage under the canonical map $G\to G/N$.  Then by the preceeding remarks, $\alpha(\widetilde{L}) = \widetilde{L'}$ for some $L'\in\mathcal{L}$.  If $L = \{x_1,\ldots,x_a\}$ and $L'=\{y_1,\ldots,y_b\}$, then we have $\bigcup_{i=1}^a\alpha(x_i)N = \bigcup_{j=1}^by_j N$.  That these are disjoint unions of distinct cosets implies that $\mathrm{Proj}^G_{G/N}(\alpha)(L) = L'$.  Thus, $T$ is $B$-characteristic.
\end{proof}


\section{The characteristic case}

Let
\[
	D_{2n} = \langle r,s\,:\,r^n=s^2=e,\,srs=r^{-1}\rangle = \langle r\rangle \rtimes\langle s\rangle.
\]
Our goal will be to construct the supercharacter theories of $D_{2n}$ from those of its cyclic normal subgroups $\langle r^d\rangle$ for divisors $d$ of $n$.  Since the subgroups $\langle r^d\rangle$ are cyclic, we may invoke Theorem \ref{thm:Zn_classification} and classify all of $\mathrm{SCT}(D_{2n})$.

The character tables of dihedral groups can be computed easily.  For odd $n$, they take the form

\begin{equation}\label{eq:odd_n_CT}
	\arraycolsep=2.4pt\def\arraystretch{1.5}
	\begin{array}{|c|c|c|c|c|c|c|}
		\hline
		\text{class rep.} & e & r^k;\;1\leq k\leq\frac{n-1}{2} & s \\
		\hline
		\text{class size} & 1 & 2 & n \\
		\hline
		\mathbf{1}_{D_{2n}} & 1 & 1 & 1 \\
		\hline
		\lambda & 1 & 1 & -1 \\
		\hline
		\chi_m,\;1\leq m\leq \frac{n-1}{2} & 2 & 2\cos\big(\frac{2\pi km}{n}\big) & 0 \\
		\hline
	\end{array}
\end{equation}

while for even $n$, they take the form

\begin{equation}\label{eq:even_n_CT}
	\arraycolsep=2.4pt\def\arraystretch{1.5}
	\begin{array}{|c|c|c|c|c|c|c|}
		\hline
		\text{class rep.} & e & r^k;\;1\leq k\leq\frac{n-2}{2} & r^{\frac{n}{2}} & s & rs \\
		\hline
		\text{class size} & 1 & 2 & 1 & \frac{n}{2} & \frac{n}{2} \\
		\hline
		\mathbf{1}_{D_{2n}} & 1 & 1 & 1 & 1 & 1 \\
		\hline
		\lambda & 1 & 1 & 1 & -1 & -1 \\
		\hline
		\mu_0 & 1 & (-1)^k & (-1)^{\frac{n}{2}} & 1 & -1 \\
		\hline
		\mu_1 & 1 & (-1)^k & (-1)^{\frac{n}{2}} & -1 & 1 \\
		\hline
		\chi_m,\;1\leq m\leq \frac{n-2}{2} & 2 & 2\cos\big(\frac{2\pi km}{n}\big) & 2\cos(\pi m) & 0 & 0 \\
		\hline
	\end{array}
\end{equation}

\begin{rmk}
	We will denote $\chi_{n/2} := \mu_0+\mu_1$ if $n$ is even, and say $\chi_{n/2}\in X$ if $\mu_0,\mu_1\in X$.  This allows us to define supercharacter theories of $D_{2n}$ in terms of the indices of the characters $\chi_m$ regardless of the parity of $n$.
\end{rmk}

\subsection{The classification}

Let $\mathcal{P}_1$ be the image of $D_{2n}\mathrm{InvSCT}(\langle r\rangle)$ under the $\ast$-product with the unique supercharacter theory $M(D_{2n}/\langle r\rangle)$ of $D_{2n}/\langle r\rangle$ and note that by Lemmas \ref{lem:starproductlatticeinjection} and \ref{lem:star_preserves_char}, $\mathcal{P}_1$ is a sublattice of $\mathrm{CharSCT}(D_{2n})$ isomorphic to $D_{2n}\mathrm{InvSCT}(\langle r\rangle)$.  Moreover, if $T = (\mathcal{K},\mathcal{X})$ is an element of $D_{2n}\mathrm{InvSCT}(\langle r\rangle)$, then we can describe its image $T\ast_{\langle r\rangle}M(D_{2n}/\langle r\rangle) = (\mathcal{L},\mathcal{Y})$ easily: the superclass partition is $\mathcal{L}=\mathcal{K}\cup\{s\langle r\rangle\}$ and the supercharacters are $\mathbf{1}_{D_{2n}}$ and $\lambda$ along with $\mathrm{Ind}_{\langle r\rangle}^{D_{2n}}(\sigma_X)$ for $X$ in $\mathcal{X}\setminus\{\mathbf{1}_{\langle r\rangle}\}$.  One can also check directly that the superclass partition of $mm_{\langle r\rangle}(D_{2n})$ is
\[
	\{s\langle r\rangle\}\cup\Big\{\{r^k,r^{n-k}\}\,:\,0\leq k\leq\left\lfloor\frac{n}{2}\right\rfloor\Big\}
\]
(so $mm_{\langle r\rangle}(D_{2n})=m(D_{2n})$ if $n$ is odd).  Finally, $MM_{\langle r\rangle}(D_{2n})$ has the superclass partition $\{\{e\}, s\langle r\rangle, \langle r\rangle\setminus\{e\}\}$.  Thus, it follows by Lemma \ref{lem:starproductlatticeinjection} that a supercharacter theory factors over $\langle r\rangle$ if and only if it contains $s\langle r\rangle$ as a superclass.

For each divisor $d$ of $n$, denote
\begin{equation}\label{eq:Fd_def}
	F_d = m_{D_{2n}}(\langle r^d\rangle)\ast_{\langle r^d\rangle}M(D_{2n}/\langle r^d\rangle)
\end{equation}
and
\begin{equation}\label{eq:Cd_def}
	C_d = M_{D_{2n}}(\langle r^d\rangle)\ast_{\langle r^d\rangle}M(D_{2n}/\langle r^d\rangle),
\end{equation}
and let $\mathcal{P}_d = [F_d,C_d]$ be the interval of supercharacter theories lying between $F_d$ and $C_d$.  If $d = 1$, then $F_d$ and $C_d$ reduce to the minimal and maximal $\ast$-products $mm_{\langle r\rangle}(D_{2n})$ and $MM_{\langle r\rangle}(D_{2n})$, respectively and so our definition of $\mathcal{P}_1$ is not redundant.  If $d = n$, then $F_d$ and $C_d$ coincide and are equal to $M(D_{2n})$.

If $n$ is even, let $\mathcal{Q}$ be the subposet of $\mathcal{P}_1$ consisting of those supercharacter theories for which $\chi_{n/2}$ is a supercharacter (otherwise, let $\mathcal{Q} = \emptyset$).  The classification of $\mathrm{CharSCT}(D_{2n})$ is given by Theorem \ref{thm:D2n_char_classification}.

\begin{thm}\label{thm:D2n_char_classification}
	If $n$ is odd, then we can express the characteristic supercharacter theories of $D_{2n}$ as a disjoint union of the form
	\[
		\mathrm{CharSCT}(D_{2n}) = \bigsqcup_{d\mathrel{\vert}n}\mathcal{P}_d.
	\]
	If $n$ is even, then for any $T\in\mathcal{Q}$, we may produce a new supercharacter theory $\psi(T)\in\mathrm{CharSCT}(D_{2n})$ by splitting the two conjugacy classes of reflections into distinct superclasses.  In this case, we may express $\mathrm{CharSCT}(D_{2n})$ as a disjoint union of the form
	\[
		\mathrm{CharSCT}(D_{2n}) = \bigg(\bigsqcup_{d\mathrel{\vert}n}\mathcal{P}_n\bigg) \sqcup \psi(\mathcal{Q}).
	\]
\end{thm}

\subsection{Proofs}

If $S\in\mathrm{SCT}(D_{2n})$ and $s\langle r\rangle$ is a (possibly proper) subset of an $S$-superclass, then we will say that $S$ \emph{glues reflections}.

\begin{lem}\label{lem:form_of_char}
	A supercharacter theory $S$ of $D_{2n}$ (for any $n$) is characteristic if and only if either $s\langle r\rangle$ is a union of $S$-superclasses or $S$ glues reflections.
\end{lem}

\begin{proof}
	Let $S$ be a characteristic supercharacter theory.  For $i=0,1$, let $K_i$ be the superclass containing $sr^i\langle r^2\rangle$.  If $\tau$ is the automorphism defined by $r\mapsto r$ and $s\mapsto sr$, then $\tau$ transposes $s\langle r^2\rangle$ and $sr\langle r^2\rangle$, hence $\tau$ transposes $K_0$ and $K_1$.  Because $\tau$ fixes all rotations, it follows that either $K_0=K_1$, or $K_i=sr^i\langle r^2\rangle$ for $i=0,1$.  In the former case, $s\langle r\rangle$ is a subset of the superclass $K_0$, while in the latter case, $s\langle r\rangle$ is the union of the superclasses $K_0$ and $K_1$.

	Conversely, write $S=(\mathcal{K},\mathcal{X})$ and suppose $s\langle r^2\rangle$ and $sr\langle r^2\rangle$ are parts of $\mathcal{K}$.  Then $n$ is necessarily even, since these are separate conjugacy classes.  Now, $\mu_0$ and $\mu_1$ lie in different parts of $\mathcal{X}$, say $X_0$ and $X_1$, respectively.  Let $\tau$ be as before, and note that because $\tau$ permutes $s\langle r^2\rangle$ and $sr\langle r^2\rangle$ and fixes all rotations, it follows that $\tau\cdot\mathcal{K}=\mathcal{K}$ and therefore $\tau\cdot\mathcal{X}=\mathcal{X}$.   Thus, since $\tau$ transposes $\mu_0$ and $\mu_1$ and fixes all other characters, it follows that $X_i = \{\mu_i\}$ for $i=0,1$.  Taking the join $S\vee mm_{\langle r\rangle}(D_{2n})$ glues $X_0$ to $X_1$ and $s\langle r^2\rangle$ to $sr\langle r^2\rangle$ and preserves all other superclasses and supercharacters.  Thus, $S\vee mm_{\langle r\rangle}(D_{2n})$ factors over $\langle r\rangle$.  By Lemma \ref{lem:star_preserves_char}, it follows that $S\vee mm_{\langle r\rangle}(D_{2n})$ is characteristic.  Since $\mathrm{Aut}(D_{2n})$ fixes $X_0\cup X_1$ and $s\langle r\rangle$, it follows that $S$ is characteristic.

	Next, write $S=(\mathcal{K},\mathcal{X})$ and suppose $s\langle r\rangle$ is a subset of a superclass.  Then if $n$ is even, it follows that $\mu_0$ and $\mu_1$ lie in the same part of $\mathcal{X}$.  Write $s\langle r\rangle\cup A$ for the part of $\mathcal{K}$ containing $s\langle r\rangle$ and write $\{\lambda\}\cup B$ for the part of $\mathcal{X}$ containing $\lambda$.   We claim that $A$ is fixed under the action of $\mathrm{Aut}(D_{2n})$.  Let $r^k\in A$ and let $j$ be coprime to $n$.  Then for any part $X$ of $\mathcal{X}$ which does not contain $\lambda$ or $\mathbf{1}_{D_{2n}}$ as constituents, say $X = \{\chi_\ell:\ell\in I\}$ (where $I$ may contain $n/2$), we have
	\[
		\sigma_X(r^k) = 2\zeta_n^{k\frac{n}{2}} + \sum_{\ell\in I\setminus\{\frac{n}{2}\}}2(\zeta_n^{k\ell} + \overline{\zeta_n}^{k\ell}) = 0
	\]
	if $n/2\in I$, and otherwise
	\[
		\sigma_X(r^k) = \sum_{\ell\in I}2(\zeta_n^{k\ell} + \overline{\zeta_n}^{k\ell}) = 0.
	\]
	In both cases, the resulting equation defines a polynomial $f(x)\in\mathbb{Z}[x]$ such that $f(\zeta_n) = \sigma_X(r^k) = 0$.  Any such polynomial $f(x)$ is divisible by the $n$th cyclotomic polynomial, and therefore has $\zeta_n^j$ as a root, for any $j$ coprime to $n$.  Hence we may replace $\zeta_n$ with $\zeta_n^j$ in these equations, which yields $\sigma_X(r^{kj})=0$.  Thus, every supercharacter of $S$ which does not contain $\lambda$ or $\mathbf{1}_{D_{2n}}$ as constituents agrees on $r^k$ and $r^{kj}$.  Since $\rho_{D_{2n}}$ and $\mathbf{1}_{D_{2n}}$ agree on these elements and
	\[
		\sigma_{\{\lambda\}\cup B} = \rho_{D_{2n}} - \mathbf{1}_{D_{2n}} - \sum_{\substack{X\in\mathcal{X} \\ \chi(1)>1 \forall \chi\in X}}\sigma_X,
	\]
	it follows that every supercharacter of $S$ agrees on $r^k$ and $r^{kj}$, and so these elements lie in the same part of $\mathcal{K}$.  This implies that $A$ is fixed under the action of $\mathrm{Aut}(D_{2n})$.  Next, write $MM_{\langle r\rangle}(D_{2n}) = (\mathcal{L},\mathcal{Y})$, where
	\[
		\mathcal{L} = \big\{\{e\}, s\langle r\rangle, \langle r\rangle\setminus\{e\}\big\}
	\]
	and
	\[
		\mathcal{Y} = \big\{\{\mathbf{1}_{D_{2n}}\}, \{\lambda\}, \mathrm{Irr}(D_{2n})\setminus\{\mathbf{1}_{D_{2n}}, \lambda\}\big\}.
	\]
	Then
	\[
		\mathcal{K}\wedge\mathcal{L} = \big\{s\langle r\rangle, A\big\} \cup \big(\mathcal{K}\setminus\big\{s\langle r\rangle\cup A\big\}\big)
	\]
	and
	\[
		\mathcal{X}\wedge\mathcal{Y} = \big\{\{\lambda\}, B\big\} \cup \big(\mathcal{X}\setminus\big\{\{\lambda\}\cup B\big\}\big),
	\]
	and it is not hard to show that these partitions form a supercharacter theory, namely $S\wedge MM_{\langle r\rangle}(D_{2n})$.\footnote{One shows that $(\mathcal{K}\wedge\mathcal{L}, \mathcal{X}\wedge\mathcal{Y})$ is a supercharacter theory by showing that the characters $\lambda$ and $\sigma_B$ are constant on the parts of $\mathcal{K}\wedge\mathcal{L}$.}  This supercharacter theory factors over $\langle r\rangle$, hence it is characteristic by Lemma \ref{lem:star_preserves_char}.  Since we have already shown that $A$ is fixed under the action of $\mathrm{Aut}(D_{2n})$, and $s\langle r\rangle$ is also fixed, it follows that $\mathrm{Aut}(D_{2n})$ permutes the remaining superclases amongst themselves.  But these are precisely the superclasses of $\mathcal{K}$ not equal to $s\langle r\rangle\cup A$.  Therefore, we have shown that $S$ is characteristic.
\end{proof}

\begin{lem}\label{lem:glues_reflections}
	Every characteristic supercharacter theory of $D_{2n}$ either glues reflections, or is covered by a unique factorizable supercharacter theory which glues reflections.  
\end{lem}

\begin{proof}
	Let $S$ be characteristic.  If $S$ does not glue reflections (which can only happen if $n$ is even), then it follows by Lemma \ref{lem:form_of_char} that $s\langle r^2\rangle$ and $sr\langle r^2\rangle$ are superclasses.  Since $s\langle r^2\rangle$ and $sr\langle r^2\rangle$ lie in different superclasses, it follows that $\mu_0$ and $\mu_1$ lie in different parts of the supercharacter partition of $S$, say $X_0$ and $X_1$ respectively.  Let $\tau$ be the automorphism of $D_{2n}$ defined by $r\mapsto r$ and $s\mapsto sr$.  Then $\tau$ transposes $\mu_0$ and $\mu_1$, but it fixes all other irreducible characters.  Hence, $X_0 = \{\mu_0\}$ and $X_1=\{\mu_1\}$.  One can check that $S\vee mm_{\langle r\rangle}(D_{2n})$ factors over $\langle r\rangle$ and that $\abs{S\vee mm_{\langle r\rangle}(D_{2n})} = \abs{S} - 1$, which implies that $S\vee mm_{\langle r\rangle}(D_{2n})$ covers $S$.  This supercharacter theory also contains $\{\mu_0,\mu_1\}$ in its supercharacter partition.  That this is the unique factorizable supercharacter theory covering $S$ follows from the observation that any factorizable supercharacter theory must glue reflections, and that any supercharacter theory which glues reflections must glue $\mu_0$ and $\mu_1$.
\end{proof}

\begin{rmk}
	One can check that any factorizable supercharacter theory containing $\{\mu_0,\mu_1\}$ may be refined in a manner reverse to that of Lemma \ref{lem:glues_reflections}.  Thus, these factorizable covers are precisely the set of supercharacter theories whose supercharacter partitions each contain $\{\mu_0,\mu_1\}$ as a part, and this refinement is precisely the map $\psi$ described in Theorem \ref{thm:D2n_char_classification}.  
\end{rmk}

\begin{lem}\label{lem:complement_forms_subgroup}
	Let $S=(\mathcal{K},\mathcal{X})$ be a supercharacter theory of $D_{2n}$ which glues reflections and let $K\in\mathcal{K}$ be the part containing $s$.  Then $D_{2n}\setminus K$ is a subgroup of $\langle r\rangle$.
\end{lem}

\begin{proof}
	Clearly, $D_{2n}\setminus K$ is a subset of $\langle r\rangle$ which contains the identity and is closed under inversion.  Let $g_1,g_2\in D_{2n}\setminus K$ and let $L_1$ and $L_2$ be the parts of $\mathcal{K}\setminus\{K\}$ containing $g_1$ and $g_2$, respectively.  Then the coefficient of $s$ in $\underline{L_1}\cdot\underline{L_2}$ is zero, whence $\underline{L_1}\cdot\underline{L_2}$ lies in the span of $\{\underline{L}:L\in\mathcal{K}\setminus\{K\}\}$.  This implies that $g_1g_2$ lies in a superclass other than $K$, so $D_{2n}\setminus K$ is closed under multiplication and is therefore a subgroup of $\langle r\rangle$.
\end{proof}

Before we complete the classification, we pause now to write down the superclass and supercharacter partitions of $F_d$ and $C_d$.  Recall their definitions in \eqref{eq:Fd_def} and \eqref{eq:Cd_def}.  A lengthy calculation using \eqref{eq:star_superclasses} and \eqref{eq:star_supercharacters} yields $F_d = (\mathcal{K}_d,\mathcal{X}_d)$ and $C_d = (\mathcal{L}_d,\mathcal{Y}_d)$, where
\begin{equation}\label{eq:Fd_calc}\begin{split}
	\mathcal{K}_d &= \big\{\{e\}, s\langle r\rangle\cup\{r^k\,:\, d \mathrel{\not\vert} k\}\big\} \cup \big\{\{r^k,r^{n-k}\}\,:\,d \mathrel{\vert} k\}\big\}, \\
	\mathcal{X}_d &= \Big\{\{\mathbf{1}_{D_{2n}}\}, \{\lambda\}\cup\Big\{\chi_k\,:\,\frac{n}{d} \mathrel{\vert} k\Big\}\Big\}\cup\bigcup_{\ell=1}^{\left\lfloor\frac{n}{2d}\right\rfloor}\Big\{\Big\{\chi_k\,:\,k\equiv\pm\ell\mathrel{\mathrm{mod}} \frac{n}{d}\Big\}\Big\}, \\
\end{split}\end{equation}
\begin{equation}\label{eq:Cd_calc}\begin{split}
	\mathcal{L}_d &= \big\{\{e\}, s\langle r\rangle\cup\{r^k\,:\, d \mathrel{\not\vert} k\}, \{r^k\,:\, d \mathrel{\vert} k\}\big\}, \\
	\mathcal{Y}_d &= \Big\{\{\mathbf{1}_{D_{2n}}\}, \{\lambda\}\cup\Big\{\chi_k\,:\,\frac{n}{d} \mathrel{\vert} k\Big\}, \Big\{\chi_k\,:\,\frac{n}{d} \mathrel{\not\vert} k\Big\}\Big\}.
\end{split}\end{equation}

\begin{lem}\label{lem:glue_reflect_in_Pd}
	Suppose $S=(\mathcal{K},\mathcal{X})$ is any supercharacter theory of $D_{2n}$ which glues reflections.  Then there exists a divisor $d$ of $n$ such that $F_d\leq S\leq C_d$.
\end{lem}

\begin{proof}
	First note that if $S$ factors over $\langle r\rangle$, then $S\in\mathcal{P}_1$ and we are done.  Assume $S$ does not factor over $\langle r\rangle$.  Let $K$ be the superclass of $S$ which contains $s\langle r\rangle$, and note that by Lemma \ref{lem:complement_forms_subgroup}, $D_{2n}\setminus\{K\}$ is a subgroup of $\langle r\rangle$, which is necessarily of the form $\langle r^d\rangle$ for some divisor $d$ of $n$.  Thus, $K = s\langle r\rangle\cup\{r^k:d\mathrel{\not\vert}k\}$, and so $\mathcal{K}$ shares this superclass with the superclass partition of $F_d$.  Since this is the only $F_d$-superclass which is not a conjugacy class, it easily follows that $F_d\leq S$.
	
	Since $\langle r^d\rangle$ is $S$-supernormal, we may consider the deflated supercharacter theory $S^{D_{2n}/\langle r^d\rangle}$, as in \eqref{eq:deflated_sct}.  By inspecting the superclass partition of $S^{D_{2n}/\langle r^d\rangle}$, we observe that $S^{D_{2n}/\langle r^d\rangle} = M(D_{2n}/\langle r^d\rangle)$, which implies that the number of blocks of $\mathcal{X}$ contained in $\mathrm{Irr}(D_{2n}/\langle r^d\rangle)$ is two.  Since one of these is the trivial character $\mathbf{1}_{D_{2n}}$, it follows that the other must be $\{\lambda\}\cup\{\chi_\ell\,:\,n/d\mathrel{\vert}\ell\}$.  Thus, $\mathcal{X}$ contains two of the three blocks of $\mathcal{X}_d$.  This easily implies that $\mathcal{X}\leq\mathcal{X}_d$, and therefore $S\leq C_d$.
\end{proof}

\begin{proof}[Proof of Theorem \ref{thm:D2n_char_classification}]
	By Lemma \ref{lem:form_of_char}, a supercharacter theory $S=(\mathcal{K},\mathcal{X})$ of $D_{2n}$ is characteristic if and only if either $S$ glues reflections, or $s\langle r\rangle$ is a union of superclasses.  If $S$ glues reflections, then Lemma \ref{lem:glue_reflect_in_Pd} implies that $S\in\mathcal{P}_d$ for some divisor $d$ of $n$.  If $S$ does not glue reflections, then Lemma \ref{lem:glues_reflections} and the following remark imply that $S = \psi(T)$ for some $T\in\mathcal{Q}$.

	Finally, we claim the subposets $\mathcal{P}_d$ ($d\mathrel{\vert}n$) and $\psi(\mathcal{Q})$ are pairwise disjoint.  Since every member of $\mathcal{P}_d$ glues reflections and no member of $\psi(\mathcal{Q})$ glues reflections, it follows that $\mathcal{P}_d\cap\psi(\mathcal{Q})=\emptyset$ for all $d\mathrel{\vert}n$.  By the proof of Lemma \ref{lem:glue_reflect_in_Pd}, every element of $\mathcal{P}_d$ has $K_d = s\langle r\rangle\cup\{r^k:d\mathrel{\not\vert}k\}$ as a part of its superclass partition, which implies that $\mathcal{P}_d\cap\mathcal{P}_e=\emptyset$ for distinct divisors $d,e\mathrel{\vert}n$.  Therefore, these subposets are pairwise disjoint.
\end{proof}


\section{The non-characteristic case}

Let $\mathcal{R}$ be the subposet of supercharacter theories of $\mathrm{CharSCT}(D_{2n})$ which glue reflections and whose superclass partitions satisfy the condition that if $K$ is a superclass containing only powers of $r$, then these powers are all of the same parity (we will say these supercharacter theories \emph{respect parity}).  Equivalently, by Lemma \ref{lem:complement_forms_subgroup}, $\mathcal{R}$ is the set of supercharacter theories of the form $S=T\ast_{\langle r^d\rangle}M(D_{2n}/\langle r^d\rangle)$, where $d$ divides $n$, $T\in D_{2n}\mathrm{InvSCT}(\langle r^d\rangle)$, and $\mathrm{Res}^{D_{2n}}_{\langle r^d\rangle}(\mu_0)$ is a $T$-superclass function.  For $i=0,1$, define functions $\psi_i:\mathrm{CharSCT}(D_{2n})\to\mathrm{SCT}(D_{2n})$ as follows.  For $S=(\mathcal{K},\mathcal{X})\in\mathcal{R}$, let $\psi_i(S) = (\mathcal{L}_i,\mathcal{Y}_i)$ be the following supercharacter theory.  If $s\langle r\rangle\cup A\cup B$ is the $S$-superclass containing the reflections, where $A$ contains only even powers of $r$ and $B$ only odd powers of $r$ (and one or both may be empty), then $\psi_i$ refines $\mathcal{K}$ by distinguishing parity, i.e.,
\[
	\mathcal{L}_0 = \big\{s\langle r^2\rangle\cup A, sr\langle r^2\rangle\cup B\big\}\cup\big(\mathcal{K}\setminus\{s\langle r\rangle\cup A\cup B\}\big)
\]
and
\[
	\mathcal{L}_1 = \big\{s\langle r^2\rangle\cup B, sr\langle r^2\rangle\cup A\big\}\cup\big(\mathcal{K}\setminus\{s\langle r\rangle\cup A\cup B\}\big).
\]
The supercharacter partition is defined by removing $\mu_i$ from the part $X$ of $\mathcal{X}$ which contains $\mu_0$ and $\mu_1$, so that
\[
	\mathcal{Y}_0 = \big\{\{\mu_0\}, X\setminus\{\mu_0\}\big\}\cup\big(\mathcal{X}\setminus\{X\}\big)
\]
and
\[
	\mathcal{Y}_1 = \big\{\{\mu_1\}, X\setminus\{\mu_1\}\big\}\cup\big(\mathcal{X}\setminus\{X\}\big).
\]

\begin{lem}
	If $S=(\mathcal{K},\mathcal{X})\in\mathcal{R}$, then $\psi_i(S) = (\mathcal{L}_i,\mathcal{Y}_i)$ forms a supercharacter theory for $i=0,1$.
\end{lem}

\begin{proof}
	Assume $i=0$ (the proof for $i=1$ is identical) and let $(\mathcal{K},\mathcal{X})$ and $(\mathcal{L}_i,\mathcal{Y}_i)$ be labeled as above.  Immediately, we have that if $Y\in\mathcal{X}\setminus\{X\}$, Then $\sigma_Y$ is constant on the parts of $\mathcal{L}_i$.  By assumption, the parts of $\mathcal{K}\setminus\{s\langle r\rangle\cup A\cup B\}$ respect parity.  Thus because $\mu_0(r^k)=\mu_0(r^ks) = (-1)^k$, it follows that $\mu_0$ is constant on the parts of $\mathcal{L}_i$.  Finally, because the regular character $\rho_{D_{2n}}$ is constant on the parts of $\mathcal{L}_i$ and
	\[
		\sigma_{X\setminus\{\mu_0\}} = \rho_{D_{2n}} - \mu_0 - \sum_{Y\in\mathcal{X}\setminus\{X\}}\sigma_Y,
	\]
	it follows that $\sigma_{X\setminus\{\mu_0\}}$ is constant on the parts of $\mathcal{L}_i$.  Therefore, $(\mathcal{L}_i,\mathcal{Y}_i)$ is a supercharacter theory.
\end{proof}

Note that if $A$ and $B$ are both empty, then $\{\mu_0,\mu_1\}\in\mathcal{X}$ and $\psi_0=\psi_1=\psi$, where $\psi$ is the map defined in Theorem \ref{thm:D2n_char_classification}.  We claim that the every noncharacteristic supercharacter theory is of the form $\psi_i(S)$ for some $i=0,1$ and some $S\in\mathcal{R}$ for which $s\langle r\rangle$ is not a superclass.

\begin{thm}\label{thm:D2n_nonchar_classification}
	Let $S = (\mathcal{K},\mathcal{X})$ be a supercharacter theory which is not characteristic, and let $\tau\in\mathrm{Aut}(D_{2n})$ be the automorphism which sends $s$ to $sr$ and fixes all rotations.  Then $S\vee S^\tau\in\mathcal{R}$ and $S = \psi_i(S\vee S^\tau)$.
\end{thm}

Before we prove Theorem \ref{thm:D2n_nonchar_classification}, we digress to formulate the rules for multiplying irreducible characters of $D_{2n}$, all of which are consequences of the character tables \eqref{eq:odd_n_CT} and \eqref{eq:even_n_CT}.  First, we have
\[
	\mu_i^2 = \mathbf{1}_{D_{2n}}
\]
for $i=0,1$,
\[
	\mu_0\cdot\mu_1 = \lambda,
\]
and
\[
	\mu_i\cdot\lambda = \mu_{1-i}
\]
for $i=0,1$.  Next, let $0\leq i\leq 1$, $1\leq m< \frac{n}{2}$, $0\leq k<n$, and consider
\begin{align*}
	(\mu_i\cdot\chi_m)(r^k) &= 2(-1)^k\cos\Big(\frac{2\pi km}{n}\Big) \\
	&= 2\cos\Big(\frac{2\pi k\frac{n}{2}}{n}\Big)\cos\Big(\frac{2\pi km}{n}\Big) + 2\sin\Big(\frac{2\pi\frac{n}{2}}{n}\Big)\sin\Big(\frac{2\pi km}{n}\Big) \\
	&= 2\cos\Big(\frac{2\pi k(\frac{n}{2}-m)}{n}\Big) \\
	&= \chi_{\frac{n}{2}-m}(r^k).
\end{align*}
Since $\mu_i\cdot\chi_m$ vanishes on $s\langle r\rangle$, it follows that
\[
	\mu_i\cdot\chi_m = \chi_{\frac{n}{2}-m}.
\]
Finally, let $1\leq\ell,m<\frac{n}{2}$, $0\leq k<n$, and consider
\begin{equation}\label{eq:chi_mult}\begin{split}
	(\chi_\ell\cdot\chi_m)(r^k) &= 4\cos\Big(\frac{2\pi k\ell}{n}\Big)\cos\Big(\frac{2\pi km}{n}\Big) \\
	&= 2\Big(\cos\Big(\frac{2\pi k(\ell+m)}{n}\Big) + \cos\Big(\frac{2\pi k\abs{\ell-m}}{n}\Big)\Big).
\end{split}\end{equation}
Now, we can generalize the definition of $\chi_a$ to all integers $a$ if we define $\chi_a(r^k)=2\cos(\frac{2\pi ka}{n})$ and $\chi_a(sr^k)=0$.  In this case, $\chi_a=\chi_{-a}$ if $a$ is negative, $\chi_0=\lambda+\mathbf{1}_{D_{2n}}$, $\chi_{\frac{n}{2}}=\mu_0+\mu_1$ as before, and $\chi_a = \chi_{\frac{n}{2}-a}$ for $a>\frac{n}{2}$.  Thus, \eqref{eq:chi_mult} becomes
\[
	\chi_{\ell+m}(r^k) + \chi_{\abs{\ell-m}}(r^k).
\]
Finally, since $\chi_\ell\cdot\chi_m$ vanishes on $s\langle r\rangle$, we have
\[
	\chi_\ell\cdot\chi_m = \chi_{\ell+m} + \chi_{\abs{\ell-m}}.
\]

\begin{proof}[Proof of Theorem \ref{thm:D2n_nonchar_classification}]
	By Lemma \ref{lem:form_of_char},
	\[
		\mathcal{K} = \big\{s\langle r^2\rangle\cup A, s\langle r^2\rangle\cup B\big\} \cup \big\{K_i\,:\,i=1,\ldots,\abs{S}-2\big\},
	\]
	where at least one of $A$ or $B$ is nonempty.  Because $S$ does not glue reflections,
	\[
		\mathcal{X} = \big\{\{\mu_0\}\cup Y, \{\mu_1\}\cup Z\big\} \cup \big\{X_i\,:\,i=1,\ldots,\abs{S}-2\big\},
	\]
	and because $S$ is not characteristic, at least one of $Y$ or $Z$ is nonempty.  We may write $S\vee S^\tau = (\mathcal{K}\vee\mathcal{K}^\tau,\mathcal{X}\vee\mathcal{X}^\tau)$, where
	\[
		\mathcal{K}\vee\mathcal{K}^\tau = \big\{s\langle r\rangle\cup A\cup B\big\} \cup \big\{K_i\,:\,i=1,\ldots,\abs{S}-2\big\}
	\]
	and
	\[
		\mathcal{X}\vee\mathcal{X}^\tau = \big\{\{\mu_0,\mu_1\}\cup Y\cup Z\big\} \cup \big\{X_i\,:\,i=1,\ldots,\abs{S}-2\big\}.
	\]
	Thus, $S\vee S^\tau$ glues reflections, so this supercharacter theory is characteristic.  If we show that one of $Y$ or $Z$ is empty, then either $\mu_0$ or $\mu_1$ will be an $S$-superclass function.  In either case, this will imply that $S\vee S^\tau$ respects parity and that $S=\psi_i(S\vee S^\tau)$.  By Lemmas \ref{lem:complement_forms_subgroup} and \ref{lem:glue_reflect_in_Pd}, we know that $\{K_i:i=1,\ldots\abs{S}-2\}$ forms the superclass partition for the restricted supercharacter theory $S_{\langle r^d\rangle}$, and one of the parts of $\mathcal{X}\vee\mathcal{X}^\tau$ is equal to $\mathrm{Irr}(D_{2n}/\langle r^d\rangle)\setminus\{\mathbf{1}_{D_{2n}}\}$.  We consider two cases.

	\textbf{Case 1:}
	Suppose this part is $\{\mu_0,\mu_1\}\cup Y\cup Z$.  Then $\langle r^d\rangle \subseteq \ker(\mu_0)$, so $d$ is even.  Also, $\lambda\in Y\cup Z$, so without loss of generality we may assume $\lambda\in Y$.  We claim that $Z$ is empty.  Assume not, write $Y = \{\lambda\}\cup Y'$, and consider
	\begin{align*}
		\sigma_{\{\mu_0,\lambda\}\cup Y'}\cdot\sigma_{\{\mu_1\}\cup Z} &= (\mu_0 + \lambda + \sigma_{Y'})\cdot(\mu_1 + \sigma_Z) \\
		&= \mu_0\cdot\mu_1 + \mu_0\cdot\sigma_Z + \lambda\cdot\mu_1 + \lambda\cdot\sigma_Z + \sigma_{Y'}\cdot\mu_1 + \sigma_{Y'}\cdot\sigma_Z,
	\end{align*}
	which after expanding becomes
	\begin{equation}\label{eq:char_prod_1}
		\lambda + \sum_{\chi_m\in Z}2\chi_{\frac{n}{2}-m} + \mu_0 + \sigma_Z + \sum_{\chi_\ell\in Y'}2\chi_{\frac{n}{2}-\ell} + \sum_{\substack{\chi_\ell\in Y' \\ \chi_m\in Z}}4(\chi_{\ell+m} + \chi_{\abs{\ell-m}}),
	\end{equation}
	Note that
	\[
		Y\cup Z = \{\lambda\} \cup \Big\{\chi_m\,:\,1\leq m<\frac{n}{2},\,\frac{n}{d}\mathrel{\vert}m\Big\}.
	\]
	Moreover, $Y$ and $Z$ are disjoint sets of irreducible characters, so there are disjoint subsets $I,J\subseteq\{\frac{n}{d},\ldots,\frac{(d-2)n}{2d}\}$ such that $Y'=\{\chi_\ell:\ell\in I\}$, $Z=\{\chi_m:m\in J\}$, and $I\cup J=\{\frac{n}{d},\ldots,\frac{(d-2)n}{2d}\}$.  So because $\frac{n}{d}\mathrel{\vert}\frac{n}{2}$, it follows that $\frac{n}{2}-k\in I\cup J$ for all $k\in I\cup J$.  Thus, \eqref{eq:char_prod_1} becomes
	\begin{equation}\label{eq:char_prod_1.1}
		\sigma_{\{\mu_0\}\cup Y} + 2\sigma_Z + \sum_{\substack{\ell\in I \\ m\in J}}4(\chi_{\ell+m}+\chi_{\abs{\ell-m}}).
	\end{equation}
	Now, \eqref{eq:char_prod_1.1} is a linear combination of $S$-supercharacters by \cite[Theorem 2.2]{Diaconis2008}.  Thus since we have assumed $Z\neq\emptyset$ and $\sigma_Z$ appears in this equation with nonzero coefficient, it follows that $\mu_1$ must also appear with nonzero coefficient.  This can only happen if the index of one of the summands of the rightmost sum in \eqref{eq:char_prod_1.1} is $n/2$.    Hence, there exist $\ell\in I$ and $m\in J$ with $\ell+m=n/2$.  Hence we have $4\chi_{\frac{n}{2}}$ appearing in the rightmost sum of \eqref{eq:char_prod_1.1}.  But $4\chi_{\frac{n}{2}}=4\mu_0+4\mu_1$, so this implies that the coefficient of $\mu_0$, and therefore $\sigma_{\{\mu_0\}\cup Y}$, in \eqref{eq:char_prod_1.1} is at least 5.  Thus, the coefficient of $\lambda$ in this equation is at least 5, and therefore there exist $\ell\in I$ and $m\in J$ with $\abs{\ell-m}=0$.  But this is impossible because $I$ and $J$ are disjoint, so we have derived a contradiction.  Therefore, $Z$ is empty.

	\textbf{Case 2:}
	Now assume that one of the $X_i$ is equal to $\mathrm{Irr}(D_{2n}/\langle r^d\rangle)\setminus\{\mathbf{1}_{D_{2n}}\}$; without loss, assume it is $X_1$, so that
	\[
		X_1 = \{\lambda\}\cup\Big\{\chi_m\,:\,\frac{n}{d}\mathrel{\vert}m\Big\}.
	\]
	Then $d$ is odd and, with notation as in \eqref{eq:star_supercharacters}, $\{\mu_0, \mu_1\}\cup Y\cup Z$ is of the form $W_0^{D_{2n}}$, where $W_0$ is some part of the supercharacter partition $\mathcal{W}$ of $S_{\langle r^d\rangle}$.  Thus, by Frobenius reciprocity, $\mathrm{Res}^{D_{2n}}_{\langle r^d\rangle}(\sigma_{\{\mu_0,\mu_1\}\cup Y\cup Z})$ is equal to $\sigma_{W_0}$ up to scalar multiplication by a positive integer $C$.  Because $\mathcal{K}$ contains the parts $\{K_i:i=1,\ldots,\abs{S}-2\}$, which are the superclasses of $S_{\langle r^d\rangle}$, it follows that $\sigma_{\{\mu_0\}\cup Y}$ and $\sigma_{\{\mu_1\}\cup Z}$ are constant on these parts.  Thus, we have nonnegative integers $c_W$, $d_W$ for $W\in\mathcal{W}$ such that
	\[
		\mathrm{Res}^{D_{2n}}_{\langle r^d\rangle}(\sigma_{\{\mu_0\}\cup Y}) = \sum_{W\in\mathcal{W}}c_W\sigma_W
	\]
	and
	\[
		\mathrm{Res}^{D_{2n}}_{\langle r^d\rangle}(\sigma_{\{\mu_1\}\cup Z}) = \sum_{W\in\mathcal{W}}d_W\sigma_W.
	\]
	But
	\begin{align*}
		\sum_{W\in\mathcal{W}}(c_W+d_W)\sigma_W &= \mathrm{Res}^{D_{2n}}_{\langle r^d\rangle}(\sigma_{\{\mu_0\}\cup Y}) + \mathrm{Res}^{D_{2n}}_{\langle r^d\rangle}(\sigma_{\{\mu_1\}\cup Z}) \\
		&= \mathrm{Res}^{D_{2n}}_{\langle r^d\rangle}(\sigma_{\{\mu_0,\mu_1\}\cup Y\cup Z}) \\
		&= C\cdot \sigma_{W_0}.
	\end{align*}
	Hence, it follows that $c_W=d_W=0$ unless $W=W_0$, so the sets of irreducible constituents of the restrictions of $\sigma_{\{\mu_0\}\cup Y}$ and $\sigma_{\{\mu_1\}\cup Z}$ are both equal to $W_0$.

	If $\xi_m$ denotes the character of $\langle r^d\rangle$ given by $r^d\mapsto \zeta_{\frac{n}{d}}^m$, then
	\begin{align*}
		\mathrm{Res}^{D_{2n}}_{\langle r^d\rangle}(\sigma_{\{\mu_0\}\cup Y}) &= \mathrm{Res}^{D_{2n}}_{\langle r^d\rangle}(\mu_0) + \sum_{\chi_\ell\in Y}2\mathrm{Res}^{D_{2n}}_{\langle r^d\rangle}(\chi_m) \\
		&= \xi_{\frac{n}{2d}} + \sum_{\chi_m\in Y}2(\xi_m+\overline{\xi_m}) \\
		&= c_{W_0}\sum_{\xi_m\in W_0}\xi_m.
	\end{align*}
	Thus, $Y$ can only contain $\chi_m$ if $m\equiv\frac{n}{2d}\mathrel{\mathrm{mod}}\frac{n}{d}$.  Similarly, $Z$ can only contain characters of this form, and any character of this form must lie in $Y\cup Z$, so that
	\[
		Y\cup Z = \Big\{\chi_m\,:\,m\equiv\frac{n}{2d}\mathrel{\mathrm{mod}}\frac{n}{d}\Big\}.
	\]

	As with the previous case, we claim that one of $Y$ or $Z$ is empty.  Suppose both $Y$ and $Z$ are nonempty, write $Y=\{\chi_\ell:\ell\in I\}$ and $Z=\{\chi_m:m\in J\}$, where $\{(2k-1)\cdot\frac{n}{2d}:1\leq k\leq\frac{d-1}{2}\}=I\sqcup J$, and consider
	\begin{align*}
		\sigma_{\{\mu_0\}\cup Y}\cdot\sigma_{\{\mu_1\}\cup Z} &= (\mu_0 + \sigma_Y)\cdot(\mu_1 + \sigma_Z) \\
		&= \mu_0\cdot\mu_1 + \mu_0\cdot\sigma_Z + \sigma_Y\cdot\mu_1 + \sigma_Y\cdot\sigma_Z,
	\end{align*}
	which, after expanding becomes
	\begin{equation}\label{eq:char_prod_2}
		\lambda + \sum_{m\equiv\frac{n}{2d}\mathrel{\mathrm{mod}}\frac{n}{d}}2\chi_{\frac{n}{2}-m} + \sum_{\substack{\ell\in I \\ m\in J}}4(\chi_{\ell+m} + \chi_{\abs{\ell-m}}).
	\end{equation}
	Consider the first sum in \eqref{eq:char_prod_2}.  If $m\equiv\frac{n}{2d}\mathrel{\mathrm{mod}}\frac{n}{d}$, then $\frac{n}{2}-m\equiv0\mathrel{\mathrm{mod}}\frac{n}{d}$.  Thus, \eqref{eq:char_prod_2} becomes
	\begin{equation}\label{eq:char_prod_3}
		\sigma_{X_1} + \sum_{\substack{\ell\in I \\ m\in J}}4(\chi_{\ell+m}+\chi_{\abs{\ell-m}}).
	\end{equation}
	Let $\ell\in I$ and $m\in J$.  Then $\ell+m$ and $\abs{\ell-m}$ are both nonzero integers less than $n$ and $\ell+m,\abs{\ell-m}\equiv0\mathrel{\mathrm{mod}}\frac{n}{d}$.  Thus because $I$ and $J$ are nonempty, the rightmost sum in \eqref{eq:char_prod_3} is nonempty, and every summand belongs to $X_1$.  Hence, the coefficient of $\sigma_{X_1}$ in \eqref{eq:char_prod_3} is at least 3.  Thus, the coefficient of $\lambda$ in this equation is at least 3, and therefore there exist $\ell\in I$ and $m\in J$ with $\abs{\ell-m}=0$.  But because $I$ and $J$ are disjoint, $\ell$ and $m$ are always distinct.  Therefore we have derived a contradiction, so one of $Y$ or $Z$ -- without loss assume $Z$ -- is empty.

	In both cases, we have shown that $Z$ is empty, so we may write
	\[
		\mathcal{X} = \big\{\{\mu_0\}\cup Y, \{\mu_1\}\big\} \cup \big\{X_i\,:\,i=1,\ldots,\abs{S}-2\big\}.
	\]
	Thus because $\mu_1$ is an $S$-supercharacter, it follows that any two rotations $r^k$ and $r^\ell$ lie in the same superclass only if $k\equiv\ell\mathrel{\mathrm{mod}} 2$, and that
	\[
		A = \big\{r^k\,:\,d\mathrel{\not\vert}k,\,2\mathrel{\not\vert}k\big\}
	\]
	and
	\[
		B = \big\{r^k\,:\,d\mathrel{\not\vert}k,\,2\mathrel{\vert}k\big\}.
	\]
	Therefore $S = \psi_1(S\vee S^\tau)$.
\end{proof}


\bibliographystyle{alpha}
\bibliography{books,library}

\end{document}